\documentclass[11pt]{amsart}

\usepackage{amsmath}
\usepackage{amsfonts}
\usepackage{amssymb, amscd}
\usepackage{graphics}
\usepackage{epsf}
\usepackage{epsfig}
\usepackage{url}

\theoremstyle{plain}

\newtheorem{theorem}{Theorem}[section]
   \newtheorem{lemma}[theorem]{Lemma}

    \newtheorem{corollary}[theorem]{Corollary}
   \newtheorem{proposition}[theorem]{Proposition}

   \theoremstyle{definition}
   \newtheorem{definition}[theorem]{Definition}
   \newtheorem{example}[theorem]{Example}
   
   \newtheorem{remark}[theorem]{Remark}
   \newtheorem{notation}[theorem]{Notation}

\newcommand{\QQ}{\mathbb Q}
\def\bQ{\mathbb Q}
\def\bZ{\mathbb Z}
\newcommand{\TT}{\mathbb T}
\newcommand{\ZZ}{\mathbb Z}

\newcommand{\PP}{\mathbb P}
\newcommand{\bA}{\mathbb A}
\newcommand{\bP}{\mathbb P}

\newcommand{\cT}{\mathcal{T}}
\newcommand{\cA}{\mathcal{T}}
\newcommand{\cV}{\mathcal{V}}
\newcommand{\cN}{\mathcal{N}}

\def\cW{\mathcal{W}}
\def\cF{\mathcal{F}}
\def\cG{\mathcal{G}}

\def\cH{\mathcal{H}}

\def\cL{\mathcal{L}}
\def\ov{\overline}
\def\eps{\varepsilon}
\def\Ord{\operatorname{\bf ord}}
\def\ord{\operatorname{\bf ord}}
\def\val{\operatorname{\bf val}}
\def\ox{\overline x}
\def\id{\operatorname{id}}
\def\ind{\operatorname{index}}

\def\Ker{\operatorname{Ker}}

\def\mbA{{\mathbf A}}
\def\bL{{\mathbb L}}
\def\oX{{\overline X}}
\def\oE{{\overline E}}

\begin{document}

\title[Elimination Theory for Tropical Varieties]
{Elimination Theory for Tropical Varieties}

\author[Bernd Sturmfels and Jenia Tevelev]
{Bernd Sturmfels and Jenia Tevelev}

\begin{abstract}
Tropical algebraic geometry offers new tools
for elimination theory and implicitization.
We determine the
tropicalization of the image of a subvariety of
an algebraic torus under any homomorphism
from that torus to another torus.
\end{abstract}

\maketitle

\section{Introduction}

Elimination theory is the art of computing the image
of an algebraic variety under a morphism. 
Our set-up is as follows.
Let $\TT^n$ be the
$n$-dimensional algebraic torus over an algebraically
closed field $k$.
Given any closed subvariety
$X \subset \TT^n$ and any homomorphism of tori
$\alpha : \TT^n \rightarrow \TT^d$, our
objective is to compute the Zariski closure
of the image of $X$ under $\alpha$. 
Abusing notation, we denote this closure by $\alpha(X)$.

In tropical elimination theory,
the varieties $X\subset\TT^n$ and $\alpha(X)\subset\TT^d$ are replaced by their 
tropicalizations. The {\em tropicalization} of $X$ is the following subset of $\QQ^n$:
\begin{equation}\label{deftropvar}
 \mathcal{T}(X)=\bigl\{v \in \QQ^n\ |\ 1 \not\in {\rm in}_v(I_X) \bigr\}.
\end{equation}
Here $I_X$ is the ideal of $X$ 
in the Laurent polynomial ring $k[\TT^n]$, and 
${\rm in}_v(I_X) \subset k[\TT^n]$ is the ideal of all
initial forms ${\rm in}_v(f)$ for $f\in I_X$. The set
$\cT(X)$ has the structure of a tropical variety
of dimension $\dim_k (X)$. The following definition clarifies what this means for us:
an {\em (abstract) tropical variety} is a pair $(\cT,m)$,
where $\cT\subset\bQ^n$ and $m:\,\cT^0\to\bZ_{>0}$ is a locally constant function 
(called {\em multiplicity}) which satisfies:
\begin{itemize}
\item There exists a pure-dimensional rational polyhedral fan supported on $\cT$.
\item $\cT^0\subset\cT$ is the open subset of {\em regular points}, where $v\in \cT$ is called regular if
there exists a vector subspace $\bL_v\subset\bQ^n$ such that $\cT=\bL_v$ locally near $v$.
\item The function $m$ satisfies the {\em balancing condition} (see 
Definition~\ref{defbalance}) for one (and hence for any) fan supported on the set $\cT$.
\end{itemize}

There is generally no canonical (or coarsest) fan structure on $\cT(X)$; see
Example~\ref{nonconvexex}.
However, $\cT(X)$ carries a poset of {\em tropical fans}
that have desirable algebro-geometric properties \cite{HKT,Te}. Example~\ref{redherring} 
shows that not all fans on $\cT(X)$ are tropical in the
sense of  Definition~\ref{def:tropfan}.
A particular tropical fan on $\mathcal{T}(X)$, arising from the Gr\"obner fan
of a homogenization of $I_X$, was used for the algorithms in \cite{BJSST}.

The multiplicity $m_v$ of a point $v$ in $\cT(X)^0$ can be computed 
as the sum of the multiplicities of all minimal associate primes
of the initial ideal $\text{in}_v (I_X)$. This recipe was
proposed in \cite[\S 2]{DFS} and it satisfies the balancing
condition by \cite[\S 2.5]{DSPhD}.
A self-contained derivation of the multiplicities on
tropical fans will be given in Section~\ref{multsection}.

Returning to our tropical elimination theory, 
let $\cT_1\subset\bQ^n$ and $\cT_2\subset\bQ^d$ be abstract tropical varieties.
We say that a map  $f:\,\cT_1\to\cT_2$ 
is {\em generically finite of degree} $\delta$ if 
\begin{itemize}
\item $f$ is the restriction of a linear map $\mbA:\,\bZ^n\to\bZ^d$;
\item $f$ is surjective and $\dim\cT_1=\dim\cT_2$;
\item for any point $w\in\cT_2^0$ such that 
$f^{-1}(w)\subset\cT_1^0$ and $|f^{-1}(w)|<\infty$, we have
\begin{equation}\label{superduperformula}
m_w \,\,= \,\, \frac{1}{\delta}\sum_{v \in f^{-1}(w)} \!\! m_y\cdot\ind(\bL_w\cap\bZ^d:\mbA(\bL_v\cap\bZ^n)).
\end{equation}
\end{itemize}

We shall prove the following general result about
morphisms of tropical varieties.

\begin{theorem}\label{funnyform}
Let  $\alpha:\,\TT^n\to\TT^d$ be  a homomorphism of tori
and let $\mbA:\, \bZ^n\to\bZ^d$ be the corresponding linear map of lattices of one-parameter subgroups.
Suppose that $\alpha$  induces a generically finite morphism 
 of degree $\delta$ from $X$ onto $\alpha(X)$.
Then $\mbA$ induces a generically finite map
of tropical varieties from $\cT(X)$ onto $\cT(\alpha(X))$.
\end{theorem}

\begin{remark}
Theorem \ref{funnyform} constitutes a refinement of the known identity
\begin{equation}\label{Talpha=AT}
\mathcal{T}(\alpha (X)) \quad = \quad\mbA(\mathcal{T}(X)).
\end{equation}
Equation (\ref{Talpha=AT}) appeared in different guises in \cite[Proposition 3.1]{Te} 
and in \cite[Theorem 3.1]{DFS}. What is new here is the
formula (\ref{superduperformula}) for the push-forward of multiplicities.
\end{remark}

This paper rests on the foundations of 
tropical algebraic geometry which were laid by
Hacking, Keel, and Tevelev in \cite{HKT,Te}.
We review these foundations in Section~\ref{Geometric Tropicalization}.
A particularly important ingredient is the
{\em normal crossing construction} in Theorem~\ref{tropfornccmain}
which characterizes tropical varieties in terms of resolutions of singularities.

In Section 3 we discuss multiplicities on tropical varieties,
and we highlight the connections to intersection theory on toric varieties.
To be precise, we equate the balancing condition with the axiom for
Minkowski weights introduced by
Fulton and Sturmfels \cite{FS}.
Theorem \ref{funnyform} will then be 
restated and proved in Theorem \ref{mainmult}.

In Section 4 we examine the case when $X$ is a generic complete intersection.
This case was also studied by Khovanskii and Esterov \cite{KE},
and we extend some of their results.
In our view, the mixed volumes in \cite{KE}
 are best thought of as
multiplicities on  tropical varieties.
Theorem \ref{main4} gives a formula for  the tropical complete intersection $\mathcal{T}(X)$
and its multiplicities in terms of mixed volumes. Our push-forward formula
for complete intersections (Corollary~\ref{MVeeCor})
is particularly interesting when $\alpha(X)$ is a hypersurface, in which
case it computes  McMullen's {\em mixed fiber polytopes} \cite{McM}.

A special case of elimination is  {\em implicitization}, which 
transforms a parametrization of an algebraic variety
into its representation as the zero set of polynomials.
To model implicitization, we take $X$ to be the graph of the
parametrization and $\alpha $ the projection onto the image coordinates.
This leads to the implicitization formula in Theorem \ref{multFormula}. This
formula was announced in our paper with Yu \cite{STY}.
Its proof is now completed in Section 5 below.
Our results generalize the earlier work on $A$-discriminants by 
Dickenstein, Feichtner and Sturmfels \cite{DFS}, who studied the
tropical implicitization problem for Kapranov's Horn uniformization.
Software for {\bf tr}opical {\bf im}plicitization, which offers an implementation of
Theorem \ref{multFormula}, is described in \cite{HY}.

We close the introduction with an explicit
example which illustrates Theorem \ref{funnyform}.

\begin{example} \label{ex:curve}
Let $X$ denote the curve in $\TT^3 $
defined by the two equations
\begin{equation}
\label{twoeqns}
 x^3+y^3+z^3 \,=\, 5 \qquad \hbox{and} \qquad
x^{-2} + y^{-2} + z^{-2} \,=\, 7.
\end{equation}
We compute the image of $X$ under the map
$\alpha : \TT^3 \rightarrow \TT^2, \, (x,y,z) \mapsto (u,v) \,$ given by
$$ u= xyz\,\,\hbox{and}\,\, v  = yz^2 \quad \hbox{or,\,\, in matrix notation,}
 \quad  {\mathbf A}   \,\,\, = \,\,\, \begin{pmatrix}1 & 1 & 1 \\ 0 & 1 & 2 \end{pmatrix}. $$
 We demonstrate how one constructs the Newton polygon of the plane curve
$\alpha(X) \subset \TT^2$ prior to knowing its equation.
Theorem \ref{main4} tells us that the
 tropicalization $\mathcal{T}(X)$
of the space curve $X \subset \TT^3$ is the one-dimensional
fan consisting of the six rays spanned by $\pm (1,1,0)$, $\pm (1,0,1)$, $\pm (0,1,1)$,
where each ray has multiplicity~$6$. 
By \eqref{Talpha=AT}, the tropical curve $\mathcal{T}(\alpha(X))$ consists
 of the rays spanned by the six columns~of
$$ 6 \cdot
 \mathbf{A} \cdot  \begin{pmatrix}
 1 &  1 & 0 &  -1 &  -1 & \phantom{-}0 \\
 1 &  0 & 1 & -1 &  \phantom{-}0 & -1 \\
 0 &  1 & 1 &  \phantom{-}0 & -1 & -1 \\
 \end{pmatrix}
\quad = \quad
\begin{pmatrix}
12 & 12 & 12 & -12 & -12 & -12 \\
 6  & 12 & 18 &   -6 &  -12 & -18
\end{pmatrix}.
$$
We multiply each vector by $6$ as a way of recording the information
that each ray has multiplicity~$6$.
To obtain the Newton polygon of the plane curve $\alpha(X)$,
we rotate the six vectors by $90$ degrees, and concatenate them
to form a hexagon with vertices
\begin{equation} \label{hexagon}
(0, 36),\,\,
(6, 24),\,\,
(18, 12), \,\,
 (36, 0),\,\,
 (30, 12),  \,\,
 (18, 24). 
 \end{equation}
 Each of the $213$ lattice points in this hexagon contributes
 one term to the equation:
 $$
 u^{36} \, + \, v^{36}
\, + \,15625 u^6 v^{24}
\, + \, 15625 u^{18} v^{12}
 - \,40353607 u^{30} v^{12}
 - \,40353607u^{18} v^{24} \,+ \,\cdots.
 $$
 In Section 4 we shall see how the
 hexagon is constructed as a mixed fiber polytope.
\end{example}

\medskip

\noindent {\bf Acknowledgments}.
We are grateful to Igor Dolgachev, Paul Hacking, Sean Keel, David Speyer,
and Josephine Yu for useful discussions. 
The first author was supported by the National
Science Foundation (DMS-0456960).
The second author was supported by a Sloan Research Fellowship.

\section{Geometric Tropicalization}\label{Geometric Tropicalization}
Let $\TT $ be an algebraic torus over $k$ and $X\subset \TT$ an
irreducible closed subvariety. By {\em geometric tropicalization} of $X$
we mean the characterization of $\mathcal{T}(X)$ given in \cite{HKT,Te}
in terms of constructions of algebraic geometry.
We start out by reviewing the connections between 
the geometric tropicalization and descriptions 
of $\cT(X)$ using valuations of the coordinate ring of $X$ and degenerations of
 $X$ inside torus $\TT$.

\begin{notation}
Throughout this paper we fix the following notation related to toric varieties.
We write $M$ for the lattice of characters of $\TT$ and
$N:=M^\vee$ for the dual lattice. 
The tropical variety $\cT(X)$ will live in $N_\QQ$.
For any fan $\cF\subset N_\QQ$, we denote by $\bP(\cF)$ the corresponding toric variety
and by $X(\cF)$ the closure of $X$ in~$\bP(\cF)$.
\end{notation}

To streamline our logic, our point of departure will be the valuative
definition of $\mathcal{T}(X)$, called the \emph{Bieri--Groves set} in \cite{EKL},
and not the ideal-theoretic definition \eqref{deftropvar}.
Let $K/k$ be the field of Puiseux series with parameter $\eps$ and
with the valuation
$${\Ord} : K \rightarrow \QQ\cup\{\infty\},\quad
\alpha \eps^u + (\hbox{\rm higher order terms}) \,\mapsto \, u.$$
Let  $K[X] = k[X] \otimes_k K$. A {\em ring valuation}
$\val:\,K[X] \to\QQ\cup\{\infty\}$ is, by definition, any map that can be written as
a composition $v\circ f$, where $f:\,K[X] \to L$ is a homomorphism to a field
and $v:\,L\to \QQ\cup\{\infty\}$ is a field valuation.

\begin{definition}[\cite{BG}]
Let $\cV(X)$ be the set of valuations of $K[X]$ that restrict
to $\Ord$ on~$K$. Any $\val \,\in\, \cV(X)$
specifies an element $[\val]$ of $N_\QQ = {\rm Hom}(M,\QQ)$
by formula
$$[\val](m):=\val(m|_X)\quad\text{for any}\quad m\in M.$$
We define the tropical variety of $X$ set-theoretically as follows:
$$\cA(X):=\bigl\{[\val]\,|\,\val\in\cV(X)\bigr\}\subset N_\QQ.$$
\end{definition}

The next two theorems emphasize that not all valuations in $\cV(X)$ are
needed but we may restrict to valuations
defined by germs of curves or to divisorial valuations.

\begin{theorem}[Einsiedler--Kapranov--Lind \cite{EKL}]
Any $K$-valued point $\gamma\in X(K)$
defines a valuation on $K[X]$
by the formula 
$\val_\gamma(f):=\ord f(\gamma)$ and we have
$$\cA(X)= \bigl\{[\val_\gamma]\,|\,\gamma\in X(K)\bigr\}\subset N_\QQ.$$
\end{theorem}

\begin{theorem}[Hacking-Keel-Tevelev {\cite[\S2]{HKT}}] \label{CorOfMain}
Let $\cW(X)$ be the set of divisorial discrete valuations
of the function field $k(X)$,
i.e.~valuations of the form $c \cdot \val_D$, where $c\in\bQ$
and $\val_D$ is the order of zero-poles along an irreducible divisor $D$ on a normal variety birationally
isomorphic to $X$.
Then
$$\cA(X) =  \bigl\{ [v]\,|\,v\in \cW(X)\bigr\}\subset N_\QQ.$$
Instead of $\cW(X)$, one can use the set of all discrete valuations of $k(X)$ trivial on $k$.
\end{theorem}

In \cite{HKT}, Theorem~\ref{CorOfMain}  is deduced from Theorem~\ref{tropfornccmain} 
and Proposition \ref{surj} using resolution of singularities (or alteration of singularities in prime characteristic). We here give a more straightforward proof 
which does not require the use of resolution of singularities.
Our argument is based on the following 
characterization of $\cA(X)$.

\begin{lemma}[{\cite[Lemma 2.2]{Te}}]\label{troptoric}
Let $w \in N \backslash \{0\}$ and let $\cF$ be the one-dimensional fan in $N_\bQ$
  with just one maximal cone $\QQ_{\ge0}w$.
 Let   $D$ the unique toric divisor on $\bP(\cF)$.
 Then $w$ lies in the tropical variety $\cA(X)$ if and only if $D$ intersects
$X(\cF)$.
\end{lemma}

\begin{proof}[Proof of Theorem~\ref{CorOfMain}]
Consider any discrete valuation $v$ of the function field $k(X)$ trivial on $k$.
We claim that $v$ induces a ring valuation $\mathbf{v}$ of
$K[X]$
by the formula
\begin{equation}\label{funnyeq}
\mathbf{v}\bigl(\sum_{q\in\QQ}
a_q\eps^q\bigr)=\min_{q \in \QQ}\{v(a_q)+q\},
\end{equation}
where $a_q\in k[X]$.
There exist  $p_1,\ldots,p_s \in K$ and $b_1,\ldots,b_s \in k[X]$ such that
$$\sum_{q\in\bQ}a_q (x) \eps^q=\sum_{i=1}^s p_i(\eps)b_i(x),$$
This identity shows that
 any $a_q$ is a $k$-linear combination of
$b_1,\ldots,b_s$
and therefore $v(a_q)\ge\inf_{i=1}^s v(b_i)$.
It follows that $\mathbf{v}$ is well-defined.
The proof that $\mathbf{v}$ is a valuation goes literally as in
\cite[VI.10.1]{Bo}.
It follows from \eqref{funnyeq} that the restriction of $\mathbf{v}$ to $K$
coincides with $\ord$, and hence $\mathbf{v}\in\cV(X)$.
Therefore, $[v]=[\mathbf{v}]\in \cA(X)$.

It remains to prove that any point $w\in\cA(X)$ has the form $[v]$
for some divisorial discrete valuation $v \in \cW(X)$.
We use Lemma~\ref{troptoric} and its notation.
Note that $\bP(\cF)=\bA^1_{x_1}\times \TT^{n-1}_{x_2,\ldots,x_n}$
where the coordinates $x_1,x_2,\ldots,x_n$ of $\TT = \TT^n$ are chosen appropriately.
By Lemma~\ref{troptoric},  the divisor $D=\{x_1=0\}$ of
$\bP(\cF)$ intersects the closure $X(\cF)$ of $X$.
Let $\nu:\,\tilde X(\cF)\to X(\cF)$ be the normalization
and let $Z$ be an irreducible component of $\nu^{-1}(X(\cF) \cap D)$.
For $i \geq 2$, the coordinate function $x_i$ is invertible on
$\PP(\cF)$ and hence so is its restriction $\ox_i$ to $X(\cF)$
and its pull-back to $\tilde X(\cF)$.
This shows that $\val_Z(\ox_2) = \cdots = \val_Z(\ox_n) = 0$.
Since $\nu^*(\ox_1)$ vanishes on $Z$, we have
$\val_Z(x_1) = \lambda$ for some positive integer $\lambda$.
This implies that $w=\lambda \cdot [\val_Z]$ in $N = M^\vee$.
Now if we set $v:=(1/\lambda) \cdot \val_Z $ then $v\in \cW(X)$ and $[v] = w $.
\end{proof}

The description of $\mathcal{T}(X)$ using divisorial valuations becomes absolutely explicit
if $X$ is smooth and has a known compactification with normally crossing boundary:

\begin{theorem}[Hacking-Keel-Tevelev {\cite[\S2]{HKT}}] \label{tropfornccmain}
Assume that $X $ is smooth
and $\overline{X} \supset X $ is any compactification whose boundary
$D = \overline{X} \backslash X$
is a divisor with simple normal crossings.  Let $D_1,\ldots,D_m$
denote the irreducible components of $D$, and write
$\Delta_{X,D}$ for the simplicial complex on $\{1,\ldots,m\}$
with $\,\{i_1,\ldots,i_l\} \in \Delta_{X,D}\,$ if and only if
$D_{i_1} \cap \cdots \cap D_{i_l}$
is non-empty. Define $[D_i]:=[\val_{D_i}] \in N$,
 and, for any $\sigma\in \Delta_{X,D}$, let $[\sigma]$ be the cone in $N_\QQ$
 spanned by $\{[D_i]\,:\,i\in \sigma\}$. Then
\begin{equation}
\label{unionofcones}
 \cA(X) \,\,\,= \,\,\, \bigcup_{\sigma \in
\Delta_{X,D}}[\sigma].
\end{equation}
\end{theorem}

\begin{remark}\label{importantremark}
The proof in \cite{HKT} shows that $\cT(X)$ 
is contained in the right hand side of (\ref{unionofcones}) if $\oX$ is just normal, without any smoothness or normal crossings conditions.
But this containment can be strict: if too many boundary divisors pass through a point
then the right hand side can contain cones
of dimension greater than $\dim X$.
But $\cT(X)$ does not contain such cones 
because $\dim \cT(X) =  \dim X$.
\end{remark}

Later we will discuss ways of weakening the normal crossing assumptions.
But our main technique is to use Theorem~\ref{tropfornccmain} 
together with the following basic result.

\begin{proposition}[{\cite[3.2]{Te}}]\label{surj}
Consider a commutative diagram of elimination theory
\begin{equation}
\label{CommDiag}
\begin{CD}
 X @>f>> Y\\
@VVV                 @VVV \\
 \TT  @> \alpha >> \tilde \TT
\end{CD}
\end{equation}
where vertical arrows are closed embeddings in  tori,
$f$ is dominant, and $\alpha$ is a homomorphism of tori.
Let $\tilde M$ be the character lattice of $\tilde \TT$ and $\tilde N$
its dual.
Then $\alpha$ induces a $\bZ$-linear map $\mbA :\, N_\bQ \to {\tilde N}_\bQ$
and we have $\,\mbA(\cA(X))=\cA(Y)\,$ as in (\ref{Talpha=AT}).
\end{proposition}

\begin{proof}
We derive this from Theorem~\ref{CorOfMain}.
Let $y\in\cA(Y)$ and let $\tilde v$ be a discrete valuation of $k(Y)$
such that $[\tilde v]=y$. Let $v$ be a discrete valuation of $k(X)$
such that $v|_{k(Y)}=\tilde v$.
 We claim that $\mbA([v])=y$.
We have to check that $\tilde m(\mbA([v]))=\tilde m(y)$ for any $\tilde m\in\tilde M$.
This means that $v(\alpha^*m)=\tilde v(m)$, which is obvious.
\end{proof}

Combining Theorem~\ref{tropfornccmain} and~
Proposition~\ref{surj} gives the following corollary.

\begin{corollary}\label{tropforncc}
Suppose that the hypotheses of Proposition~\ref{surj} and Theorem~\ref{tropfornccmain}
hold. Then the tropical variety $\cA(Y)$ in ${\tilde N}_\bQ$
is the union of the cones $\mbA ([\sigma])$  where $\sigma$ runs over all
maximal simplices of
$\Delta_{X,D}$. The rays $\mbA ([D_i])$ of these cones
are the linear forms on $\tilde M$ given by
$\mbA([D_i]) (\chi) \,:=\,\val_{D_i}(\chi\circ f) $.
\end{corollary}

\begin{remark}[Intrinsic Torus {\cite[\S3]{Te}}]\label{intrinsictorus}
We wish to explain why only monomial maps are considered
in our elimination problem. 
Recall that an irreducible algebraic variety
$X$ is called {\em very affine}
 if it is isomorphic to a closed subvariety of an algebraic torus. Intrinsically this
means that $X$ is affine and its coordinate algebra $k[X]$ is generated by its units~$k[X]^*$.
In this case any choice of invertible generators $x_1,\ldots,x_n\in k[X]$ gives a closed embedding $X\subset \TT^n$. By a theorem of Samuel \cite{Sam}, 
the group of units $k[X]^*/k^*$ is finitely generated.
Hence any very affine variety is a closed subvariety of a canonical
{\em intrinsic torus} $\TT$ with character lattice $k[X]^*/k^*$.
Moreover, any morphism to the algebraic torus $f\,:X\to\tilde\TT$ 
is the restriction of the homomorphism
$\alpha:\,\TT\to\tilde\TT$ defined as follows:
the pull-back $\alpha^*$ is  the 
composition 
$$k[\tilde\TT]^*/k^*\to k[X]^*/k^*\simeq k[\TT]^*/k^*.$$
It follows that any elimination problem is monomial, at least in principle.
\end{remark}

All of the above results
characterize the tropical variety $\cT(X)$ only as a subset of a vector space over $\QQ$.
Next we recall the definition of tropical fans.
In Section~\ref{multsection}, this will furnish us with
 a convenient framework  for defining
multiplicities on $\cT(X)$.

\begin{definition}[\cite{Te}] \label{def:tropfan}
Let $\cF$ be any fan in $N_\QQ$ and $X $ a very affine variety in $\TT$.
We say that $\cF$~is a {\em tropical} fan for $X$ if $X(\cF)$ is proper and the multiplication map
\begin{equation}\label{multmap}
\Psi:\, \TT \times X(\cF)\,\to \,\bP(\cF)
\end{equation}
is flat and surjective. Note that this property depends on
both $X$ and $\cF$.
\end{definition}

\begin{theorem}[{\cite[1.2, 2.5]{Te}}]\label{BasicTe} Tropical fans have the
following properties:
\begin{enumerate}
\item If $\cF$ is a tropical fan (for $X$) then the support of $\cF$
equals the tropical variety $\cT(X)$. Any fan $\cF$ 
supported on $\cA(X)$ has a tropical refinement $\cF'$.
\item If $\cF$ is tropical and $\cF'$ is any refinement of $\cF$ then
$\cF'$ is tropical as well.
\item 
Let $\Gamma\subset\cF$ be a maximal-dimensional cone
and let $Z\subset\bP(\cF)$ be the corresponding toric stratum.
If $\cF$ is tropical then $\dim X(\cF)\cap Z=0$.
If, moreover, $\bP(\cF)$ is smooth then $X(\cF)$ is Cohen-Macaulay
at any point of $X(\cF)\cap Z$.
\end{enumerate}
\end{theorem}

The notion of a tropical fan is a very flexible generalization of the
Gr\"obner fan structure on $\mathcal{T}(X)$ used in \cite{BJSST,DFS,SS}.
We briefly review this connection.
To keep the notation simple, we assume from now on
that the normalizer of $X$ in $\TT$ is trivial.

\begin{theorem}[{\cite[1.7]{Te}}] \label{Groebnertrop}
Let $\cG $  be any complete fan in $ N_\bQ$.
Consider the action of  $\TT$ on the Hilbert scheme of $\bP(\cG)$ by the formula 
$t\cdot[S]=[t^{-1}S]$ for any subscheme $S\subset\bP(\cG)$.
The normalization of the orbit closure $\overline{\TT\cdot [X(\cG)]}$
in the Hilbert scheme is a toric variety for $\TT$. Let $\tilde\cF$ be the
complete fan of this toric variety.
The intersection  of $\tilde\cF$  with $\cT(X)$ is a subfan of $\tilde\cF$,
and that subfan is a tropical fan for $X$.
\end{theorem}

Now we relate Theorem~\ref{Groebnertrop} to Gr\"obner fans
using the Cox coordinate ring.
Let $m$ be the number of rays of $\cG$ and let
$f:\,\bZ^m\to N$ 
be a linear map that sends the $i$-th basis vector $e_i$ to the generator of the $i$-th ray of $\cG$.
This induces a surjective homomorphism $\TT^m\to \TT$.
Consider the standard embedding $\TT^m\subset\bA^m$ and let $U\subset\bA^m$
be a $\TT^m$-toric open subset obtained by throwing away coordinate subspaces
$\langle e_i\,:\,i\in I\rangle$ for each collection
of boundary divisors $\{D_i\,:\,i\in I\}$ with empty intersection.
Then we have a morphism of toric varieties $\pi:\,U\to\bP(\cG)$
extending the homomorphism $\TT^m\to \TT$. 
Let $R=k[\bA^m]$, the Cox ring of $\bP(\cG)$.
Let $\hat X\subset\bA^m$ be the closure of $\pi^{-1}(X)$ and let $I\subset R$ be its ideal.
The {\em Gr\"obner fan} $\hat\cF\subset\bQ^m$ is, by definition, supported 
on the first octant $\bQ^m_{\ge0}$, and two vectors
$w_1$ and $w_2$ are in the relative interior of the 
same cone if and only if the initial ideals ${\rm in}_{w_1}(I)$ and ${\rm in}_{w_2}(I)$ in $R$ are equal.
It is not hard to prove
(using the fact that initial ideals compute flat limits of one-parameter subgroups)
that $f$ induces the map of fans $\hat\cF\to\tilde\cF$.
Moreover, $\sigma=f^{-1}(f(\sigma))$
for any cone $\sigma\in\hat\cF$.
Since a refinement of a tropical fan is tropical (Theorem~\ref{BasicTe}),
it is possible to define a tropical fan supported on $\cT(X)$ by taking
images of cones in $\hat\cF$ that intersect (= are contained in)  $\cT(X)$.
It is this fan that was ``the tropical variety'' in the earlier papers \cite{BJSST,DFS,SS}. 
Note that this fan can be finer than the fan $\cF$ defined in Theorem~\ref{Groebnertrop}
because initial ideals ${\rm in}_wI$ can have embedded components
supported on the unstable locus $\bA^m \backslash U$.

 \section{Push-Forward of Multiplicities}\label{multsection}

We now turn our attention to the multiplicities attached to
a tropical variety. A tropical fan structure will used
to define multiplicities, but that definition turns out to be
equivalent to the one given in the introduction
(see Corollary \ref{cor:assprimes}) and thus independent of
the particular choice of tropical fan. Our main result
is Theorem \ref{funnyform}, which describes what happens
to the multiplicities under a morphism 
 $\alpha:\,\TT^n\to\TT^d$.
 
We retain the notation from Section~2, in particular, $X$
is a closed subvariety of an algebraic torus $\TT$ and $\cA(X)\subset N_\bQ$ is its
tropicalization.
If $\cF\subset N_\bQ$ is any fan, with corresponding toric variety
$\bP(\cF)$, then $X(\cF)$ denotes the closure of $X$ in $\bP(\cF)$.

\begin{definition}\label{defmult}
Suppose that $\cF$ is a tropical fan for $X$. Let $\Gamma\subset\cF$
be a maximal-dimensional cone (i.e.~of dimension $\dim (X)$)
and let $Z\subset\bP(\cF)$ be the torus orbit that corresponds to
$\Gamma$.
Note that the scheme $X(\cF)\cap Z$ is $0$-dimensional by
Theorem~\ref{BasicTe}.
We define the {\em multiplicity} $m_\Gamma$ 
of the cone $\Gamma$ to be the length of the scheme $X(\cF)\cap Z$.
\end{definition}

\begin{lemma}\label{propmult}
Let $\cF$ be any tropical fan for $X$.
\begin{enumerate}
\item If\/ $\bP(\cF)$ is smooth then $m_\Gamma$ equals the
intersection number $\deg([Z]\cdot [X(\cF)])$.
\item If\/ $\cF'$ is any fan that refines the given tropical fan $\cF$ then $m_{\Gamma'}=m_\Gamma$
for any pair of maximal-dimensional
cones $\Gamma'$ in $\cF'$ and $\Gamma$ in $\cF$ such that
$\,\Gamma'\subseteq \Gamma$.
\end{enumerate}
\end{lemma}

\begin{proof}
Statement (1) follows from \cite[Proposition~7.1]{Fu}.
Recall that, in general, the intersection number of a proper
intersection can be less than the length of an intersection subscheme.
However, such bad intersections do not happen in our case since $X(\cF)$ is Cohen-Macaulay at any
point of $X(\cF)\cap Z$,  by Theorem~\ref{BasicTe} (3).

To prove (2), let $Z\subset\bP(\cF)$ and $Z'\subset\bP(\cF')$
be the torus orbits that correspond to
$\Gamma$ and $\Gamma'$ respectively.
We consider the commutative diagram
$$
\begin{CD}
\TT\times X(\cF') @>{\Psi'}>> \bP(\cF')\\
@V{g:=(\id,f|_{X(\cF')})}VV  @VVfV \\
\TT\times X(\cF) @>{\Psi}>> \bP(\cF)\\
\end{CD}
$$
where $\Psi$ and $\Psi'$ are multiplication morphisms~\eqref{multmap},
and $f$ is the proper morphism of toric varieties that corresponds to the refinement of
fans. This is a fiber diagram by \cite[Proposition~2.5]{Te}.
Let $p'\in Z'$ be any point.
Since $\Psi$ is flat, we have
\begin{equation}
\label{cycleIdentity}
g_* \bigl({\Psi'}^*([p']) \bigr) \,\,\,= \,\,\, \Psi^* \bigl( f_*([p']) \bigr),
\end{equation}
where we use push-forward and flat pull-back  of cycles as in
\cite[1.4 and 1.7]{Fu}.
Let $p=f(p')$, and let $H$ and $H'$ denote the
stabilizers  in $\TT$ of $p$ and $p'$ respectively.
Since $\bP(\cF')$ and $\bP(\cF)$ are normal,
the subgroups $H$ and $H'$  are subtori of $\TT$.
Since $\Gamma$ and $\Gamma'$ have the same dimension,
namely ${\rm dim}(X)$, we conclude that $H= H'$.

It follows from the $\TT$-invariance of the multiplication map
$\Psi$ that the scheme-theoretic fiber $(\Psi)^{-1}(p)$
is isomorphic to $\,H\times (X(\cF)\cap Z)$. This product
is a torus times a zero-dimensional scheme.
Similarly,
the scheme-theoretic fiber $(\Psi')^{-1}(p')$
is isomorphic to $\,H\times (X(\cF')\cap Z')$, with
the same torus $H$.
The identity (\ref{cycleIdentity}) implies that
the length $m_\Gamma$ of $X(\cF)\cap Z$ coincides with
the length $m_{\Gamma'}$ of $X(\cF')\cap Z'$. \
\end{proof}

\begin{definition}[Fulton--Sturmfels \cite{FS}]\label{defbalance}
Given a rational polyhedral cone $\sigma\subset N_\bQ$,
we write $N_\sigma$ for the sublattice of $N$ generated by $\sigma\cap N$.
For a facet $\tau$ of $\sigma$ 
let $n_{\sigma,\tau}$ be any representative in $\sigma$ for the generator of the $1$-dimensional
lattice $N_\sigma/N_\tau$.
A multiplicity function $m_\sigma$ 
defined on $k$-dimensional cones of some fan $\cF$ 
is said to satisfy
the {\em balancing condition} if, for any cone $\tau$ of $\cF$ of dimension~$k-1$,
we have 
$$\sum_{\sigma\supset\tau}m_\sigma \cdot n_{\sigma,\tau}\in N_\tau$$
Here the sum over all $k$-dimensional cones $\sigma$ in $\cF$ that contain $\tau$.
\end{definition}

\begin{corollary}\label{balancingcondition}
If $\cF$ is a tropical fan for $X$ whose
toric variety $\bP(\cF)$ is smooth then
 the multiplicity function $m_\sigma$
in Definition \ref{defmult}
satisfies the balancing condition.
\end{corollary}

\begin{proof}
Let $k=\dim X$. Since we wish to cite \cite{FS}, we fix
a complete strictly simplicial fan $\cG$ that contains $\cF$ as a subfan.
Since $X(\cF)$ is proper, $X(\cF)$
does not intersect toric strata parametrized by cones in $\cG \backslash \cF$. 
We set $m_\sigma = 0$ for $k$-dimensional cones $\sigma$ of $\cG$
that are not in $\cF$.
Lemma~\ref{propmult} (1) implies that
the multiplicities $m_\sigma$ determine a class in the
{\em operational Chow cohomology}
$A^k(\bP(\cG))$, representing the intersection with $X(\cF)$.
Therefore $m_\sigma$ satisfies the balancing condition by \cite[Theorem 2.1]{FS}. 
\end{proof}

\begin{remark}
From this perspective, a tropical variety is the same as an
operational Chow cohomology class on an appropriate toric variety. 
See also \cite[Remark 1.7]{KP}.
\end{remark}

\begin{lemma}\label{nicelemma}
Suppose $p\in\cT(X)$ is a regular point (see Introduction) and suppose $\cF$ is 
a fan supported on $\cT(X)$ and endowed with a multiplicity function $m_\sigma$ that
satisfies the balancing condition.
Then $m_\sigma$ is constant on all cones $\sigma$ containing $p$.
\end{lemma}

\begin{proof}
Restricting to $\bL_p$ and taking the link $\cG$ in $\cF$ of the cone 
containing~$p$
in its relative interior, we reduce to the following statement:
If $\cG$ is a complete fan and its top-dimensional cones are 
endowed with multiplicities that satisfy the balancing condition
then these multiplicities are constant. This holds because the 
 operational Chow cohomology group $A^0$ of a
complete toric variety $\bP(\cG)$  has rank one.
\end{proof}

\begin{definition} \label{defmult2}
The {\em multiplicity} of a regular point $p\in\cT(X)$ as the multiplicity
$m_\sigma$ of any top-dimensional cone $\sigma$ of any tropical fan
that contains $p$ in its closure. 
\end{definition}

\begin{corollary}
The multiplicity is a well-defined locally constant function on the open subset of regular points of 
$\cT(X)$. For any  (not necessarily tropical) fan $\cF$ supported on $\cT(X)$,
the induced multiplicity function satisfies the balancing condition.
\end{corollary}

\begin{proof}
This follows from Lemma~\ref{nicelemma}, Corollary~\ref{balancingcondition}, Lemma~\ref{propmult},
and the fact that any two tropical fans have a common strictly simplicial refinement.
\end{proof}

The balancing condition implies that 
many complete fans can play the role of the
Gr\"obner fans in Theorem \ref{Groebnertrop}
when computing tropical varieties set-theoretically:

\begin{proposition}\label{skeleton}
Let $\cG\subset N_\bQ$ be any complete fan such that $X(\cG)$ does not intersect
toric strata of $\bP(\cG)$ of codimension greater than $k=\dim X$.
Then $\cT(X)$ equals the union of all $k$-dimensional
cones in $\cG$ whose toric strata intersect~$X(\cG)$.
\end{proposition} 

\begin{proof}
Let $\cF$ denote the subfan of $\cG$ consisting of all
cones whose toric strata intersect~$X(\cG)$. By
Theorem~\ref{BasicTe}, there exists a
tropical fan $\tilde\cF$ for $X$ which admits a map of fans $\tilde\cF\to\cG$.
By \cite[Lemma 2.2]{Te}, the support $|\tilde\cF| = \cT(X)$ is contained in $|\cF|$ and intersects the
relative interior of every cone in $\cF$. Let $C$ be a $k$-dimensional cone
in $\cF$ and let $\tilde\cF_C$ be the set of cones in $\tilde\cF$ contained in $C$.
Note that $\tilde\cF_C$ is a $k$-dimensional fan. We claim that its support
 $|\tilde\cF_C|$ equals $C$. Arguing by contradiction, suppose that it is not.
Then there exists a $k$-dimensional cone $\sigma$ in $\tilde\cF_C$ and 
a facet  $\tau$ of $\sigma$  such that
 $\tau$ intersects the interior of $C$
and lies in the boundary of $|\tilde \cF_C|$.
But this contradicts the balancing condition, and
we conclude  $|\cF| = \cT(X)$.
\end{proof}

\begin{example}\label{redherring}
The fan $\cF$ in the proof of Proposition~\ref{skeleton}
is not necessarily tropical. We construct a variety~$X \subset \TT$
that admits a non-tropical fan  $\cF$ supported on $\cT(X)$.
Consider a projective variety $Y \subset \bP^{r-1} $
which is not locally Cohen-Macaulay  at some point $p \in Y$.
Then $k:=\dim X \geq 2$.
Let $H_1,\ldots,H_k\subset\bP^{r-1}$ be general hyperplanes passing through $p$
and let $H_{k+1},\ldots,H_r\subset\bP^{r-1}$ be general hyperplanes.
The intersection of $\cap_{i\in I}H_i$ with $ Y$ is non-empty 
if $|I|=k$ and empty if $|I|>k$.
Then $X:=Y \backslash \{H_1,\ldots,H_r\}$ is a very affine variety with intrinsic
torus $\TT=\bP^{r-1}\backslash \{H_1,\ldots,H_r\}$.
Let $\cF$ be a fan obtained by taking the $k$-skeleton of the fan of $\bP^{r-1}$.
By Proposition~\ref{skeleton}, $\cF$ is supported on $\cT(X)$.
If $\cF$ is a tropical fan then $Y$ has to be Cohen-Macaulay at $p$
by Theorem~\ref{BasicTe} (3). Therefore $\cF$ is not tropical.
\end{example}

\begin{remark}
This example has to be compared with \cite[Theorem 1.9]{HKT}
which concerns the (Mori-theoretically significant) class of {\em h\"ubsch}
very affine varieties.
 If $X$ is h\"ubsch then any fan supported on $\cT(X)$
is tropical, and there exists a coarsest fan
supported on $\cT(X)$.
The latter property also fails in general, see Example~\ref{nonconvexex}.
\end{remark}

We now explain how multiplicities
behave under a map of tropical varieties
as discussed in the Introduction.
Our setting is that of the commutative diagram \eqref{CommDiag},
where vertical arrows are closed embeddings in algebraic tori,
$f : X \rightarrow Y $ is dominant, and $\alpha$ is the homomorphism of tori
specified by a $\bZ$-linear map ${\bf A}:\, N_\bQ  \to \tilde N_\bQ$.
By Proposition~\ref{surj}, we have $\,A(\cA(X))=\cA(Y)$,
which is the identity in (\ref{Talpha=AT}).
To prove Theorem~\ref{funnyform}, it suffices to prove the following more 
technical statement.

\begin{theorem}\label{mainmult}
Let $\cF$ be a fan in $N_\bQ$ whose support equals $\cA(X)$,
and let $\tilde \cF$ be a fan in $\tilde N_\bQ$ whose support equals $\cA(Y)$.
We further assume the following conditions:
\begin{itemize}
\item The map $f$ is generically finite of degree $\delta$ (hence
$\dim\tilde\cF=\dim\cF$).
\item The fan $\cF$ is tropical for $X$ and the fan $\tilde \cF$ is tropical for $Y$.
\item For any cone $\Gamma$ of $\cF$, the image ${\bf A}(\Gamma)$ is a
union of cones of $\tilde \cF$.
\end{itemize}
Then the multiplicity $m_\Pi$ of any maximal-dimensional cone
$\Pi$ of $\tilde \cF$ satisfies
\begin{equation}\label{formulaformult}
m_\Pi \quad = \,\quad \frac{1}{\delta} \cdot \!\! \sum_{\Gamma \in \cF : {\bf A}(\Gamma) \supseteq \Pi} \!\!
m_\Gamma\cdot\ind(\Gamma,\Pi),
\end{equation}
 where
$\,\ind(\Gamma,\Pi)\,$ denotes the index of the sublattice of $\tilde N$ generated by
the semigroup $\,{\bf A}(\Gamma\cap N)$ inside the
sublattice generated by the semigroup $\Pi\cap \tilde N$.
\end{theorem}

\begin{proof}
Let $k:=\dim (X)$. By Lemma~\ref{propmult}~(2) we can
replace  $\tilde \cF$ by any refinement.
We will thus assume from now on
that $\tilde \cF$ is strictly simplicial, i.e.~$\PP(\tilde \cF)$ is smooth.

Our next goals are to make $\PP(\cF)$ smooth and
 the morphism $\PP(\cF)\to\PP(\tilde\cF)$
flat. These goals are incompatible but the following lemma
provides a compromise.

\begin{lemma}\label{toricraynaud}
There exists a strictly simplicial refinement $\cF'$ of
$\cF$ such that
\begin{itemize}
\item $\,{\bf A}: \cF'\to \tilde \cF\,$ is a map of fans, and
\item if $\,\Gamma$ is any $k$-dimensional cone of $\cF'$ and
$\Pi$ is the minimal cone of $\cF$
that contains $\, {\bf A}(\Gamma)$ then either $\,\dim\Pi<\dim\Gamma \,$ or
$\,\Pi\, = \, {\bf A}(\Gamma)$.
\end{itemize}
\end{lemma}

\begin{proof}[Proof of Lemma~\ref{toricraynaud}]
We construct $\cF'$ in two steps.
First we define a refinement $\cF''$ as follows.
 For any  cone in $\,\Gamma\,$ (of any dimension),
the image $\,{\bf A}(\Gamma)\,$ is a union of cones in $\tilde \cF$.
We subdivide $\Gamma$ by preimages of these cones. Since $\cF$ is
a fan, this subdivision yields a  fan $\cF''$.
Moreover, if $\,\Gamma\in\cF''\,$ and
$\,\dim\Gamma=\dim {\bf A}(\Gamma)\,$
then $\Gamma$ is strictly simplicial since
${\bf A}(\Gamma)$ is.
To finish the construction of $\cF'$,
it remains to subdivide all cones $\Gamma\in \cF''$  with
$\dim\Gamma>\dim\bar g(\Gamma)$ into strictly simplicial subcones.
But we have to make sure that the collection of new cones is still
a fan. Arguing by induction on $\dim\Gamma$,
this construction reduces to the following fact which is well-known in
geometric combinatorics:
if $\Gamma$ is a rational polyhedral cone and $\partial\cH$ is a
strictly simplicial fan supported on the boundary $\partial\Gamma$,
then there is a strictly simplicial fan $\cH$ supported on
$\Gamma$ such that the restriction of $\cH$ to
$\partial\Gamma$ equals $\partial\cH$.
\end{proof}

By Lemma~\ref{propmult} (2), to prove \eqref{formulaformult}
it suffices to prove the analogous formula
\begin{equation}\label{formulamult}
m_\Pi \quad = \quad \,  \frac{1}{\delta} \,\, \cdot \!\!\! \sum_{\Gamma \in \cF':
{\bf A}(\Gamma)=\Pi} m_\Gamma\cdot\ind(\Gamma,\Pi)
\end{equation}
for any $d$-dimensional cone $\Pi$ of $\tilde \cF$.
To simplify notation we redenote $\cF'$~by~$\cF$.

Consider the smooth 
toric varieties $\PP(\cF)$ and $\PP(\tilde \cF)$.
The map of fans $\,{\bf A}  : \cF \rightarrow \tilde \cF\,$ 
defines a morphism $\,\alpha :\,\bP (\cF) \to \bP (\tilde \cF) \,$ that extends the homomorphism
of tori $\,\alpha :\, \TT \to \tilde \TT$.
The $d$-dimensional cone $\Gamma \subset \tilde \cF$ defines
a torus orbit $\tilde Z$ of codimension $d$ in  $\PP(\tilde \cF)$.
Our hypotheses imply that the fiber
 $\alpha^{-1}(\tilde Z)$  is the disjoint union of torus orbits
  $Z_1,\ldots,Z_r$ of codimension $d$ in $\PP(\cF)$.
  These orbits correspond to the cones
$\Gamma_1,\ldots,\Gamma_r$ in $\cF$ that satisfy
$\,{\bf A}(\Gamma_i)=\Pi$. Moreover, it is known (see
e.g.~\cite[Prop.~2.7]{FS}) that
the multiplicity of the scheme $\alpha^{-1}(\tilde Z)$ along $Z_i$ is
$\,\ind(\Gamma_i,\Pi)$.

By  Theorem~\ref{BasicTe}, the closure $ X(\cF)$ of $X$ in $\PP(\cF)$ is
Cohen-Macaulay along the subscheme $W:=\, \alpha^{-1}(\tilde Z)\cap X (\cF)$.
Therefore, we have a fiber diagram
$$
\begin{CD}
W @>i'>> X(\cF) \\
@VqVV     @V \alpha VV \\
Z  @>i>> \bP(\tilde \cF)
\end{CD}
$$
where $i$ and $i'$ are regular embeddings of codimension $d$.
By Lemma \ref{propmult} we have
$\,m_\Pi=\deg([X(\tilde \cF) ]\cdot [\tilde Z])$. Since the morphism $\alpha$
is proper, the push-forward property of
\cite[Theorem 6.2 (a) and Remark 6.2.2]{Fu} implies that
$$
m_\Pi \,\,= \,\,\deg\left(\frac{1}{\delta} \alpha_*([X(\cF)])\cdot [\tilde Z]\right) \,\,\, = \,\,\,
\frac{1}{\delta}\deg(q_*([W]))  \,\,\, = \,\,\, \frac{1}{\delta}\deg [W]. $$
From the irreducible decomposition of
$\alpha^{-1}(\tilde Z)$ discussed above, we obtain
$$ \deg [W]\,\,\,= \,\,\,\sum_{i=1}^r\ind(\Gamma_i,\Pi) \cdot \deg ([\tilde X]\cdot Z_i)
\,\,\, = \,\,\, \sum_{i=1}^r\ind(\Gamma_i,\Pi) \cdot m_{\Gamma_i}.$$
This completes the proof of Theorem \ref{mainmult},
and hence also of Theorem~\ref{funnyform}.
\end{proof}

\begin{example}
The special case of our construction when $X$ is defined
by linear equations appears in \cite[\S 3]{DFS}.
Here $\cT(X)$ is the Bergman fan of a matroid.
This is applied in \cite[\S 4]{DFS} to tropicalize
Kapranov's Horn uniformization of the A-discriminant.
\end{example}

At this point we wish to note that the
definition of multiplicity used here
(Definitions \ref{defmult} and \ref{defmult2})
agrees with  the formulation  in the Introduction
(and in \cite{DFS, STY}):

\begin{corollary}[{\cite[Remark 2.1]{DFS}}] \label{cor:assprimes}
Let $X$ be a subvariety of a torus $\TT$ and $I_X$ its ideal in $k[\TT]$.
If $w$ is any regular point in the tropical variety $\cT(X)$ then
 $m_\sigma$ equals the
 sum of multiplicities of all minimal associate primes
of the initial ideal $\text{in}_w (I_X)$.
\end{corollary}

\begin{proof}
This can be derived from the proof of
part (2) in Lemma \ref{propmult}.
\end{proof}

\begin{example} \label{ex:hypersurface}
Consider an irreducible hypersurface $X$ in $\TT$
with principal ideal $I_X = \langle G \rangle$.
The  Laurent polynomial $G$
is unique up to multiplication by a unit.
Its Newton polytope $Q\subset M_\bQ$ is unique up to translation.
 For any $w\in N_\bQ$, $\text{in}_w (G) $ 
 is the sum of all terms in $G$ that belong
 the face of $Q$ at which $w$ is minimized.
 This face is denoted  ${\rm face}_w(Q)$,  and we have
 $\text{in}_w(I_X) = \langle \text{in}_w(G) \rangle$.
 The tropical hypersurface
 $\cT(X)$ is the set of $w\in N_\bQ$ such that ${\rm face}_w(Q)$
has dimension $\geq 1$, i.e.~$\cT(X)$ is a normal fan of $Q$ with top-dimensional 
cones removed. 
A point $w \in \cA(X)$ is regular when $
{\rm face}_w(Q)$ is an edge. In this case we regard
${\rm in}_w(G)$ as a univariate polynomial. Its roots
correspond to minimal associate primes of $\text{in}_w (I_X)$.
Counting multiplicities, the number $m_\sigma$ of roots
 equals the number of lattice points on the edge
${\rm face}_w(Q)$.
\end{example}

\begin{remark}\label{getNewton}
In Example \ref{ex:hypersurface}, the Newton polytope of $Q$ can be recovered 
(up to translation)
from the tropical hypersurface $\cT(X)$.
Vertices of $Q$ correspond bijectively  to connected components of $N_\bQ
\backslash \cT(X)$.
Let $u \in N_\QQ$ be a general vector in such a component.
We compute the corresponding vertex, which is ${\rm face}_u(Q)$,
using the following formula from \cite[Theorem 2.2]{DFS}.
Consider the halflines $L_i = u + \QQ_{\geq  0} e_i$ 
where $\{e_1,\ldots,e_n\}$ is some fixed basis of $N$.
Each halfline $L_i$ intersects
$\cT(X)$ in finitely many points $v$, and with
each of these we associate the nonnegative integer
$$ \bigl[N \,:\, \bZ e_i + (\bL_w \cap N) \bigr] \cdot m_v. $$
The $i$-th coordinate of the vertex is the sum of these integers,
for all  $v \in L_i \cap \cA(X)$.
\end{remark}

\section{Complete Intersections and Mixed Volumes}\label{completeintersection}

In this section we study the case when
$X$ is a generic complete intersection
in a torus $\TT=\TT^n$. Our tropical approach
recovers and refines the elimination theory based on Newton polytopes
which was developed by Khovanskii and Esterov~\cite{KE}.

\begin{notation}
Let $P_1,\ldots,P_c$ be lattice polytopes in $M_\QQ$.
Let $P:=P_1+\cdots+P_c$ be their {\em Minkowski sum} and 
write  $P_I:=\sum_{i \in I} P_i$
for any $I\subseteq\{1,2,\ldots,c\}$.
We shall always assume that $\dim P=n$.
The {\em normal fan} $\cN$ of $P$ is a complete fan in $N_\bQ$.
\end{notation}

For each $i=1,\ldots,c$, let
$f_i$ be a Laurent polynomial in
$k[\TT]$ whose Newton
polytope equals $P_i$ and whose coefficients
are generic subject to this property.
We are interested in the generic complete intersection
$X=\, \{f_1 = \cdots = f_c = 0\}$ in $\TT$.

\begin{remark}
The variety $X$ need not be irreducible because
the $P_i$ can be one-dimensional even if $n \geq 2$.
 This poses some notational difficulties since irreducibility 
 of $X$ was assumed in the previous two sections.
  However, as we shall see, all irreducible (= connected) components of $X$ have
exactly the same tropicalization and multiplicities.
Thus by $\cT(X)$
we will mean the tropicalization of any irreducible component, with multiplicities
multiplied by the number of components.
This is consistent with the interpretation of the tropical variety as a cohomology class.
\end{remark}

As before in Example \ref{ex:hypersurface},
for any $w \in N_\QQ = {\rm Hom}(M_\QQ,\QQ)$ and any polytope $Q$,
we write ${\rm face}_w(Q)$ for the face of $Q$
at which $w$ attains its minimum.
Note that 
\begin{equation}\label{facemink}
{\rm face}_w (P_I) \,= \,\sum_{i \in I} {\rm face}_w (P_i).
\end{equation}

\begin{theorem}[{\cite[Prop. 2.4]{STY}}] \label{three}
The tropical variety $\cT(X)$ is supported on a subfan of the
normal fan $\cN$ of the Minkowski sum $P = P_1 + \cdots+ P_c$.
A point $w \in N_\QQ$ lies in $\cT(X)$ if and only if
the polytope (\ref{facemink}) has dimension $\geq |I|$ for any $I \subseteq \{1,2,\ldots,c\}$.
\end{theorem}

For $w=0$ we obtain the following criterion for
when the variety $X$ is non-empty:

\begin{corollary} \label{nonemptycrit}
$X$ is not empty  if and only if
$\dim P_I\geq |I|$
for every $I\subseteq \{1,2,\ldots,c\}$.
\end{corollary}

In \cite{STY} we proved Theorem \ref{three}
using combinatorial convexity. Here we 
present an alternative derivation using following general lemma,
which is a mixed version of the
 Kodaira--Fujita Lemma \cite[VI.2.16]{Ko}
that characterizes big and nef line bundles.

\begin{lemma}\label{genKo}
Let $L_1,L_2 ,\ldots,L_c$ be globally generated line bundles on 
a $c$-dimensional integral scheme $Z$.
For  $K\subseteq \{1,\ldots,c\}$ let $Z_K$ be the image of $Z$ under the map
given by $\otimes_{i\in K} L_i$.
Then $L_1\cdot L_2 \cdot \ldots\cdot L_c\ge0$ and this number is positive if and only if
\begin{equation}\label{condo}
\dim Z_K\ge |K|\quad\text{for any}\quad K\subseteq \{1,\ldots,c\}.
\end{equation}
\end{lemma}

\begin{proof}
By the Projection Formula \cite[VI.2.11]{Ko}, we can replace $Z$ with
its resolution of singularities. So, we may assume that $Z$ is smooth.
By Bertini's Theorem, we can choose smooth divisors $E_i\in|L_i|$ that intersect
transversally. The number of intersection points equals our
intersection number $L_1\cdot \ldots\cdot L_c$.
In particular, the number is nonnegative
and it is positive if and only if $ E_1 \cap \cdots \cap E_c \not= \emptyset$.
The condition \eqref{condo} is necessary because 
$\dim Z_K< |K|$ implies $\cap_{i\in K}E_i=\emptyset$
for general divisors $E_i\in|\cL_{i}|$.
Now assume \eqref{condo} is satisfied.
We use induction on $c$. The case $c=1$ is obvious. Assume that $c>1$.
Let $x\in Z$ be a general point. By Bertini Theorem, we can assume that $x$
lies in a smooth divisor from $|\cL_c|$.
Let $Z'$ be the connected component of $Z$ that contains $x$.
It suffices to check the condition \eqref{condo} for $Z'$
and all subsets $K\subseteq \{1,\ldots,c-1\}$.
Arguing by contradiction, suppose that
$\dim Z'_K<|K|$ for some $K$.
Let $F$ be the irreducible component of the fiber of $Z\to Z_K$ that passes through $x$.
Then $\dim (F\cap Z' ) \ge c-|K|$.
By our assumptions (and genericity of $x$), it follows that
$\dim F=c-|K|$ and $F\subset Z'$.
As this holds for any component $Z'$, the map $Z\to Z_{\{c\}}$ given by $\cL_c$ contracts $F$ to a point.
Therefore, $F$ is contained in a fiber of the map $Z\to Z_K\times Z_{\{c\}}$.
Since the image of $|\cL_K|\otimes|\cL_{\{c\}}|$
in $|\cL_{K\cup\{c\}}|$ is a base-point-free linear subsystem, $F$ is contained in a fiber of the map $Z\to Z_{K\cup\{c\}}$.
This contradicts our assumption that $\dim Z_{K\cup\{c\}}\ge |K|+1$.
\end{proof}

\begin{proof}[Proof of Theorem \ref{three}]
Let $A_i$ be the set of monomials in $f_i$, each regarded
as a lattice point in $P_i$, and consider the morphism
$\TT\to\bP^{|A_i|-1}$ defined by these monomials.
Let $\tilde\cN$ be a strictly simplicial refinement of $\cN$.
There is a sequence of morphisms
\begin{equation}\label{basictoric}
\bP(\tilde\cN)\to \bP(\cN)\to \bP^{|P_i|-1}\dasharrow\bP^{|A_i|-1},
\end{equation}
where $\bP^{|P_i|-1}\dasharrow\bP^{|A_i|-1}$ is a linear projection.
We have similar toric morphisms
\begin{equation}\label{anothertoric}
\bP(\tilde\cN)\to\bP(\cN)\to \bP^{|P_K|-1}\qquad\text{for any $K\subseteq\{1,\ldots,c\}$}.
\end{equation}
The toric variety $\bP(\tilde\cN)$ is smooth and its toric boundary $\bigcup\limits_{i=1}^rD_i$
has simple normal crossings. Consider the hypersurface $E_i:= \{f_i=0 \} \subset \TT$ 
and its closure  $\ov E_i$ in  $\bP(\tilde\cN)$.
Let $\,\{\ov E_i^\alpha\}_{\alpha \in J_i} \,$ be the irreducible components of $\oE_i$.
Then $\ov E_i$ is a pullback of a general hyperplane divisor in $\bP^{|A_i|-1}$ 
for the morphism \eqref{basictoric}.
It follows from Bertini's Theorem that each $\ov E_i$ is smooth,
$E_i^\alpha\cap E_i^\beta=\emptyset$ for $\alpha\ne\beta$, and
the union of all divisors (both of toric type $D_i$ and of type $\ov E_i^\alpha$)
has simple normal crossings. 

Recall that our very affine variety $X$ is the intersection $\bigcap_{i=1}^c E_i$
in $\TT$.
It follows that its closure $X(\tilde\cN)$ in $ \bP(\cN)$
is smooth, of codimension $c$,
and intersects toric boundary divisors $D_1,\ldots,D_r$ transversally.
It follows that the multiplication map 
$$ \TT\times X(\tilde\cN)\to\bP(\tilde\cN)$$
is smooth, and in particular flat.
By Theorem \ref{BasicTe}, the tropicalization of any irreducible component 
$X'$ of $X$ is equal to the union of cones of $\tilde\cN$
that correspond to 
$c$-dimensional toric strata $Z$ which intersect the closure 
$X'(\tilde \cN)$ of $X'$ in $\bP(\tilde \cN)$.

Let $Z$ be such a toric stratum and 
 $w$ any vector in the relative interior of the 
 corresponding cone in $\tilde \cN$.
Lemma~\ref{genKo} states that $X'(\tilde\cN)$ 
intersects the toric stratum $Z$ if and only if 
the image of $Z$ under the morphism given by the sum of divisors $\sum_{i\in K}\ov E_{i}$
has dimension at least $|K|$ for all $K\subseteq\{1,\ldots,c\}$.
This image is the toric stratum in the toric variety of $P_K$ corresponding to
${\rm face}_w(P_K)$. We conclude that 
$w$ lies in $\cT(X') = \cT(X)$ if and only if
${\rm dim} ({\rm face}_w(P_K))$ is at least $|K|$, for all $K$.
\end{proof}

We now describe the multiplicities on the tropical complete intersection  $\cA(X)$
specified by $P_1,\ldots,P_c$.
The following theorem is the main result in this section.

\begin{theorem} \label{main4}
The multiplicity of $\cT(X)$ at a regular point $w$ is
 the mixed volume
  \begin{equation}
\label{mixedvol} m_w \,\,\, = \,\,\, {\rm MV}\bigl( {\rm face}_w(P_1), {\rm face}_w(P_2), \ldots,
{\rm face}_w(P_c) \bigr).
\end{equation}
Here the volume is
normalized with respect to the sublattice of $M$ orthogonal to ${\bL_w}$.
\end{theorem}

\begin{proof}
We fix a tropical fan on $\cT(X)$ which refines
the restriction of the normal fan $\cN$ of $P_1 + \cdots + P_c$.
Since both sides of the equation (\ref{mixedvol}) are locally constant
on the regular points $w$ of $\cT(X)$, it suffices to prove (\ref{mixedvol}) 
only for points $w \in \cT(X)$  that lie
in the relative interior of a maximal cone $\Gamma$ of the
chosen tropical fan structure.

Let $\TT_\Gamma$ be the $c$-dimensional torus 
corresponding to $M_\Gamma$.
The polytope ${\rm face}_w(P_i)$ is the Newton polytope
of the initial form ${\rm in}_w(f_i)$. Consider the subscheme 
\begin{equation}
\label{inisystem}
X_w:=\{{\rm in}_w(f_1)  = \cdots = {\rm in}_w(f_c) = 0\}\subset \TT.
\end{equation}
The equations ${\rm in}_w(f_i)$ are equivariant 
under the action of $\Ker[\TT\to \TT_\Gamma]$ and therefore
we can regard  (\ref{inisystem}) as a system
of $c$ equations on $\TT_\Gamma$ defining a 
zero-dimensional subscheme $\tilde X_w$ of $\TT_\Gamma$.
By Bernstein's Theorem \cite{Be}, the length of 
the generic complete intersection $\tilde X_w$ equals
the mixed volume on the right hand side of (\ref{mixedvol}).

We claim that $\tilde X_w$ coincides with the scheme-theoretic
intersection of $X(\cN)$ with the toric stratum $Z$ of $\bP(\cN)$ that corresponds to $\Gamma$.
Indeed, since $X(\cN)$ does not intersect strata of $\bP(\cN)$
of codimension $c+1$, the intersection 
$X(\cN)\cap Z$ lies in the open orbit in~$Z$.
That open orbit is isomorphic to $\TT_\Gamma$
and the defining equations of $X(\cN)\cap Z$
are precisely the $w$-initial forms of the $f_i$. 
Definition \ref{defmult} now implies 
$$ m_w = {\rm length}(X(\cN)\cap Z) = {\rm length}(\tilde X_w). $$
With this, the conclusion of the previous paragraph now
completes the proof.
\end{proof}

\begin{remark}
The multiplicity formula \eqref{mixedvol} 
makes sense for all $w \in N_\QQ$ since
the right hand side is zero unless the 
condition of Theorem~\ref{three} holds.
Thus, $\cA(X)$ is characterized as the
set of all $w $ for which the
 mixed volume in \eqref{mixedvol} is nonzero.
\end{remark}

We now apply our elimination theory of Section~\ref{multsection}
to the generic complete intersection $X$. The goal is to compute
the image of $X$ under the morphism $\alpha : \TT^n \rightarrow \TT^d$
given by a linear map $\mbA:\,  \bZ^n \rightarrow \bZ^d$.
We assume that the induced map $f :\, X \rightarrow Y \,$ is generically finite of degree $\delta$.
This assumption implies in particular that $d \geq n-c$.
The tropical variety of $Y = \alpha (X)$ in $\QQ^d$ is characterized by the following formula.

\begin{corollary} \label{MVeeCor}
The tropicalization of the image $Y = \alpha(X)$ is the image of
 $\cT(X)$ under the linear map ${\bf A}: \QQ^n \rightarrow \QQ^d$.
Its multiplicity function is given by the formula
\begin{equation}\label{MVee}
m_v \,\,= \,\, \frac{1}{\delta} \! \sum_{w\in\mbA^{-1}(v)} \!\!\!\!
 {\rm MV}\bigl( {\rm face}_w(P_1),  \ldots,
{\rm face}_w(P_c)\bigr) \cdot \ind(\bL_v\cap\bZ^d:\mbA(\bL_w\cap\bZ^n)),
\end{equation}
for all points $v\in\cT(Y)^0$ such that $\mbA^{-1}(v)\cap\cT(X)$
is finite and contained in $\cT(X)^0$.
\end{corollary}

\begin{proof}
This follows by  combining
Theorems \ref{funnyform} and \ref{main4}.
\end{proof}

These results will be of particular interest for applications when
$d = n-c+1 $. In this case  $Y = \alpha(X)$ is a hypersurface in the torus $ \TT^d$,
and one wishes to compute the irreducible Laurent polynomial
$g \in k[\TT^d]$ which vanishes on $Y$.
Using Corollary \ref{MVeeCor} we obtain
the tropical hypersurface $\cA(Y)$ in $\QQ^d$
together with its multiplicities. By applying the technique
described in Example~\ref{getNewton}, we can construct
the Newton polytope $Q$ of the unknown Laurent polynomial $g$.
Knowledge of its Newton polytope $Q$ greatly facilitates the
computation of $g$ from the given equations of $X$.
 See Example \ref{ex:curve} for an illustration of
how this works when  $n=3$, $c=2$ and $d=2$.

A construction of the Newton polytope $Q $ from
  the given Newton   polytopes $P_1,\ldots,P_c$ 
 and the linear map ${\bf A}$  was recently given
 by Khovanskii and  Esterov \cite{KE}.
Their construction is based on
what they call {\em mixed fiber bodies}.
These turn out to be equivalent to the
{\em mixed fiber polytopes} introduced
earlier by McMullen \cite{McM}. This representation 
constitutes a geometric refinement of Corollary
\ref{MVeeCor}, and it leads to a 
simplified derivation  (in Theorem \ref{KhoEst} below)
 of the main results in \cite{KE}.

Let $\pi$ be the canonical map from the lattice $\ZZ^n$ onto
${\rm coker}({\bf A}^T)=\ZZ^n/{\rm ker}({\bf A})^\perp $. We use
the same letter $\pi$ also for
 the induced map of vector spaces
$\QQ^n \rightarrow {\rm coker}({\bf A}^T)_\QQ$.
The kernel of $\pi$ is the vector space dual
to ${\rm ker}({\bf A})$, and we identify it with $\QQ^d$ above.

Consider the Minkowski sum
$P_\lambda=\lambda_1 P_1 + \cdots + \lambda_c P_c$
where $\lambda = (\lambda_1,\ldots,\lambda_c)$
is a vector of positive unknowns. The map $\pi$
projects $P_\lambda$ onto the $(n-d)$-dimensional polytope $\pi(P_\lambda)$.
Almost all fibers of this map are polytopes of dimension $d$, and, by
integrating them in the Minkowski sense of \cite{BS}, we obtain
a $d$-dimensional {\em fiber polytope} 
$\Sigma_\pi(P_\lambda)\subset{\rm ker}(\pi)=\QQ^d$.
In fact, $\Sigma_\pi(P_\lambda)$ is 
a homogeneous polynomial of degree $n-d+1$ whose
coefficients are certain lattice polytopes $Q_{i_1 i_2 \cdots i_c}$
in $\QQ^d$:
\begin{equation}
\label{decompo}
\,\, \Sigma_\pi(P_\lambda) \,=\, \Sigma_\pi (\,\lambda_1 P_1 + \cdots + \lambda_c P_c \,)
\,\,\,\, = \sum_{i_1 + \cdots + i_c = n-d+1} \!\!\!\!\!\!\!\!\!
\lambda_1^{i_1}
\lambda_2^{i_2} \cdots \lambda_c^{i_c} \cdot Q_{i_1 i_2 \cdots i_c}.
\end{equation}
This decomposition is due to McMullen \cite{McM}, and it was rediscovered in
 \cite[\S 3]{KE}.
Now, if $n-d+1 = c$ then the formula (\ref{decompo}) defines the {\em mixed fiber polytope}
\begin{equation}
\label{MFP}\Sigma_\pi(P_1,P_2, \ldots,P_c):= Q_{11\cdots 1}.
 \end{equation}
 It is  this polytope which is computed by  our
Corollary \ref{MVeeCor} in the hypersurface case:

\begin{theorem}[Khovanskii and  Esterov \cite{KE}] \label{KhoEst}
If $c=n-d+1$  then the Newton polytope of the hypersurface
$\,Y  = \alpha(X) \subset \TT^d \,$ equals the mixed fiber polytope (\ref{MFP}).
\end{theorem}

\begin{proof}
We consider the Minkowski weight \cite{FS} of codimension $c$ defined 
by the $n$-dimensional polytope $\lambda_1 P_1 + \cdots + \lambda_c P_c$.
As a tropical variety, this is the union of the cones of codimension $\geq c$
in the normal fan of this polytope, with weights
$$ \,w \,\,\mapsto \,\, {\rm Vol}({\rm face}_w(P_1) + \cdots + {\rm face}_w(P_c)) . $$
Here ${\rm Vol} $ is the normalized 
$c$-dimensional volume on $\bL_w^\perp$.
The image in ${\rm ker}(\pi)^* = \QQ^d$
of this $c$-dimensional fan under the linear map ${\bf A}$
 is the union of the codimension one cones in
the normal fan of the fiber polytope
$\,\Sigma_\pi (\,\lambda_1 P_1 + \cdots + \lambda_c P_c) $.
It supports the Minkowski weight of codimension $1$
of $\,\Sigma_\pi (\,\lambda_1 P_1 + \cdots + \lambda_c P_c) $, which is given by
$$ v \,\, \mapsto \!
\sum_{w\in\mbA^{-1}(v)} \!\!\!\!
 {\rm Vol}\bigl( \lambda_1 {\rm face}_w(P_1) +  \cdots +
 \lambda_c
{\rm face}_w(P_c)\bigr) \cdot \ind(\bL_v\cap\bZ^d:\mbA(\bL_w\cap\bZ^n)). $$
This expression is a homogeneous polynomial of degree $c$ in
the unknowns $\lambda_1,\ldots,\lambda_c$. The coefficient
of the squarefree monomial
 $\lambda_1 \lambda_2 \cdots \lambda_c$ in this polynomial equals
the right hand side of (\ref{MVee}), up to the global constant $\delta$.
This proves Theorem 3.2 in \cite{KE}, which states that
 the mixed fiber body actually exists as a polytope,
 and we conclude that
$(1/\delta) \cdot \Sigma_\pi(P_1,P_2, \ldots,P_c) $
 equals the Newton polytope of the hypersurface $Y$.
 \end{proof}
 
 \begin{example} \label{ex:curve2}
 In Example \ref{ex:curve} we are given
 two tetrahedra in three-space, namely
 $\, P_1 =  {\rm conv}\{ 0, 3 e_1, 3 e_2, 3 e_3\} \,$ and $\,
 P_2 =  {\rm conv}\{ 0, -2 e_1, -2 e_2, -2 e_3\}$.
  Their Minkowski sum $P_1 + P_2$ is a $3$-polytope
  with $12$ vertices, $24$ edges and $14$ facets.
 The map
 $ \pi : \QQ^3 \rightarrow \QQ^1$ is given in coordinates by $ (x,y,z) \mapsto x-2y+z $.
 The fiber polytope 
 $$ \Sigma_\pi( \lambda_1 P_1 + \lambda_2 P_2) \,= \,
 \lambda_1^2 \cdot Q_{20} \,+\, \lambda_1 \lambda_2 \cdot Q_{11}
 \, +\, \lambda_2^2 \cdot Q_{02} $$
 is a polygon with ten vertices. Its summands
  $Q_{20} $ and  $Q_{02} $ are quadrangles,
  while the mixed fiber polytope
  $Q_{11} = \Sigma_\pi(P_1,P_2)$ is
  a hexagon, affinely isomorphic to
  (\ref{hexagon}).
 \end{example}
 
 \begin{example} 
 The Newton polytopes of resultants in \cite{NPR}
 are mixed fiber polytopes.
 Here $f_1,\ldots,f_c$ are Laurent polynomials
 in $c - 1$ variables whose coefficients are 
 distinct unknowns.
 See also \cite[\S 6]{DFS} for the case of
  fewer than $c-1$ variables.
 \end{example}

 We close this section with the remark that
 (ordinary) fiber polytopes
are special instances of mixed fiber polytopes.
Suppose that $P_1 = P_2 =  \cdots = P_c$ are
all equal to the same fixed polytope $P$ in $\QQ^n$.
Then the polytope $\Sigma_\pi (P_\lambda)$
 in (\ref{decompo}) equals 
$$ \Sigma( \lambda_1 P_1 + \cdots + \lambda_c P_c )
\, = \,(\lambda_1 + \cdots + \lambda_c)^c \cdot \Sigma_\pi(P) \,
\subset \, \QQ^d .$$
Hence the mixed fiber polytope
$ \Sigma_\pi(P, \cdots,P) $ is
the fiber polytope $\Sigma_\pi(P)$
times $1/c !$.

We conclude that the study of fiber polytopes,
which is an active area in geometric
combinatorics,
can be regarded as a very special
case of tropical elimination theory.

\begin{corollary} \label{KhoEst2}
If $X$ is a generic complete intersection in $\TT^n$
defined by $n-d+1$ Laurent polynomials that share the 
 same full-dimensional Newton polytope $P$,
 then the Newton polytope of the hypersurface
 $ Y= \alpha(X) \subset \TT^d$ is
 the fiber polytope $\Sigma_\pi(P)$.
 \end{corollary}

\section{Tropical Implicitization}\label{Tropical Implicitization}

We now turn our attention to the problem of implicitization~\cite{STY}.
Let $f_1,\ldots,f_s$ be Laurent polynomials in
unknowns $t_1,\ldots,t_r$
and consider the rational map
$$f= (f_1, \dots, f_s):\,\TT^r \dashrightarrow \TT^s$$ 
defined by these Laurent polynomials.
Let $Y$ be the closure of the image of $f$
in~$\TT^s$. For simplicity, we assume from the beginning
that the fiber of $f$ over
a generic point of $Y$ is finite, and consists of $\delta$ points.
Computing $Y$ is generally a hard problem
in computer algebra.
Our goal is to compute instead the
tropical variety $\cA(Y)$,
with the aim of recovering (information
about) the variety $Y$ from its
tropicalization~$\cA(Y)$.

Let $P_1,\ldots,P_s \subset \QQ^r $ be the Newton polytopes
of $f_1,\ldots ,f_s$. We first consider the generic case, when $f_i$ has
generic coefficients relative to its Newton polytope $P_i$,
and later we address the case when
the $f_i$ have special coefficients.
A rule for computing $\cA(Y)$ from
$P_1,\ldots,P_s$ was stated (with proof) in \cite[Theorem 2.1]{STY},
and a formula for the multiplicities on $\cA(Y)$
was stated (without proof) in \cite[Theorem~4.1]{STY}.
In Theorem \ref{multFormula} below we give a complete
derivation for both of these results.
Let  
$$\Psi=(\Psi_1,\ldots,\Psi_s) :\, \QQ^r \rightarrow \QQ^s$$ be the tropicalization of the
map $f$. Its $i$th coordinate $ \Psi_i(w) = \text{min}\{w \cdot v : v \in P_i\} $ is the
support function of $P_i $.
The image of $\Psi$ is contained in $\cA(Y)$ but this containment is usually strict.
Our construction will explain
 the set difference $\,\cA(Y) \backslash {\rm image}(\Psi)$.

 Let $C$ be any cone in the normal fan of
 the Minkowski sum  $P=P_1+\cdots+P_s$
 and $J$ any subset of $\{1,\ldots,s\}$.
 Consider the sublattice of $\bZ^s$ spanned by
 $\,\Psi(C \cap \bZ^r) + \bZ^J$.
 If the rank of this sublattice is $r$ then
  ${\rm index}(C,J)$ denotes its index
in the maximal rank $r$ sublattice of $\bZ^s$ that contains it.
Otherwise we set $\,{\rm index}(C,J)= 0$.
Write  $\, {\rm face}_C(P_j)\,$
for the face of $P_j$  at which the linear forms in
 the relative interior of $C$ are minimized.
 The $|J|$-dimensional normalized mixed volume
 \begin{equation}
\label{MV2} {\rm MV}\bigl( {\rm face}_C(P_j)  \,: \, j \in J \bigr)
\end{equation}
is positive if and only if
$\,{\rm dim}(\sum_{i \in K} {\rm face}_C(P_i)) \geq |K|\,$
for all subsets $K \subseteq J$.

\begin{theorem}\label{multFormula}
The tropical variety $\cA(Y)$ is the union of the cones
$\Psi(C)+\mathbb{R}_{\geq 0}^J$ over all pairs
$(C,J)$ such that
(\ref{MV2}) is positive. The multiplicity
$m_w$ at any regular point $w$ of $\cA(Y)$
is the sum of the scaled mixed volumes
\begin{equation}
\label{MV3}
\frac{1}{\delta} \cdot
 {\rm index}(C,J) \cdot
{\rm MV}\bigl( {\rm face}_C(P_j)\,:\,j\in J\bigr)
\end{equation}
where $(C,J)$ runs over all pairs such that
$\, \Psi(C) \,+ \, \mathbb{R}_{\geq 0}^J \,$ contains $w$.
\end{theorem}

\begin{proof}
The graph of the map $f$ is a
generic complete intersection
of codimension $s$ in $\TT^{r + s}$.
Namely, writing $(t_1,\ldots,t_r, y_1,\ldots,y_s)$
for the coordinates on $\TT^{r + s}$, by
the graph of $f$ we mean the subvariety
$X$ of $\TT^{r + s}$ which is defined by the equations
$$ f_1(t) - y_1 \,= \, f_2(t) - y_2 \, = \,\cdots \, =\,f_s(t) - y_s \,=\,0.$$
These are $s$ generic Laurent polynomials in $r+s$ unknowns.
Their Newton polytopes in $\QQ^r \oplus \QQ^s$
have the form
$\, P'_i \, := \, {\rm conv}(P_i \cup \{e_i\})\,$
for $i=1,2,\ldots,s$,
where $e_1,\ldots,e_s$ are the standard basis
vectors of $\QQ^s$.
We shall apply the results of Section 4
(with $d=s$ and $n=s+t$) to the
polytopes $P_1',\ldots,P_s'$
and to  the linear map ${\bf A}$ which
takes vectors $(u,v)$ in $\,\QQ^r \oplus \QQ^s\,$
to their second component $v \in \QQ^s$.
The corresponding homomorphism of tori is
$\alpha : \TT^{r+s} \rightarrow \TT^s , (t,y) \mapsto y $
and we have $Y = \alpha(X)$.

The tropical variety $\cA(Y)$ equals
the set of all
vectors $v \in \QQ^s$ such that
\begin{equation}
\label{MVee5}
{\rm MV}\bigl({\rm face}_{(u,v)}(P_1'),
\ldots,{\rm face}_{(u,v)}(P_s') \bigr)
\end{equation}
 is positive for some $u \in \QQ^r$.
Suppose this is the case where $(u,v)$ is a
regular point of $\cA(X)$. 
 Let $C$ be the cone in the normal fan of $P$ which contains
  $u$ in its relative interior.  
The vector $v - \Psi(u)$ is non-negative, for
if its $i$-th coordinate were negative then
$\,{\rm face}_{(u,v)}(P_i') $ would be a point $\{e_i\}$.
Let $J$ be the set of all indices $j$ with
 $\,v_j > \Psi_j(u)$. For $j \in J$ we
 have $\, {\rm face}_{(u,v)}(P_j') \, = \,{\rm face}_u(P_j) = {\rm face}_C(P_j)$.
  For $\,i \in \{1,\ldots,s\}\backslash J$, 
  ${\rm face}_v(P_i)$ must be a single point $p_i$
  in $\bQ^r$, since  $(u,v) $ is regular on $\cA(X)$, and hence
  $\,{\rm face}_{(u,v)}(P_i') \,$ is the primitive line segment with vertices $p_i$ and $e_i$.
 Since the lattice spanned by these segments is a direct summand of $\bZ^{r+s}$,
   the normalized mixed volume in (\ref{MVee5})
   remains unchanged if we extract these segments:
     \begin{equation}
  \label{MVee6}
  {\rm MV}\bigl({\rm face}_{(u,v)}(P_1'),
\ldots,{\rm face}_{(u,v)}(P_s') \bigr)  \, = \,
 {\rm MV}\bigl( {\rm face}_C(P_j)  \,: \, j \in J \bigr). 
 \end{equation}
 The linear space $\bL_{(u,v)}$ for the neighborhood
 of $(u,v)$ on $\cT(X)^0$
is spanned by the vectors $(u',v')$ 
where $u'$ runs over $C $ and
$v'$ satisfies  $\,v'_i = \Psi(u')_i \,$ for $i \in  \{1,\ldots,s\}\backslash J$.
  This implies that ${\bf A}( \bL_{(u,v)} \cap \bZ^{r+s})$ 
  equals the lattice generated by
  $\,\Psi(C \cap \bZ^r) + \bZ^J$.
  The maximal rank $r$ sublattice of $\bZ^s$ that contains this lattice
  is $\bL_v \cap \bZ^s$. We conclude
   \begin{equation} \label{indexissame}
     [\bL_v \cap \bZ^s:{\bf A}( \bL_{(u,v)} \cap \bZ^{r+s})]
\,=\,    {\rm index}(C,J) .
\end{equation}
Theorem \ref{multFormula} now follows by
combining (\ref{MVee6}) and (\ref{indexissame})
with Corollary \ref{MVeeCor}.
\end{proof}

\begin{example}\label{nonconvexex}
Consider the matrix 
$\,(m_{ij} ) =
\left[\begin{matrix}
1& 0& 0& 0& 1& 1\cr
0& 1& 0& 1& 0&-1\cr
0& 0& 1&-1&-1& 0\cr 
\end{matrix}\right]
\,$
and let $f:\,\TT^3\dashrightarrow\TT^6$ be a rational map
defined by six Laurent polynomials, with generic coefficients, with Newton polytopes
given by the six parallelopipeds
$$P_i \,=\,[m_{1i},2]\times[m_{2i},2]\times[m_{3i},2].$$
Let $Y=\overline{f(\TT^3)}$.
A simple calculation using Theorem~\ref{multFormula}
(that can be performed using the software TrIm developed by Huggins and Yu~\cite{HY})
 shows that $\cT(Y)$ contains 
the cone $C$ which is spanned by $e_1$, $e_2$, and $e_3$,
and the only other cone intersecting
its interior is the cone $D$ spanned by rows of the matrix $(m_{ij})$.
Moreover, $C\cap D$ is the ray spanned by $e_1+e_2+e_3$.
It follows that the set of regular points of $\cT(Y)$ is not convex. In particular,
there does not exist a coarsest fan supported on $\cT(Y)$.
\end{example}

We now discuss the implicitization problem for maps
$f$ which are not necessarily generic. The $s$
 hypersurfaces $E_i:= \{f_i=0 \} \subset\TT^r$ are assumed to be reduced, irreducible, and different.
We focus on computing $\cT(Y)$ as a set, leaving 
the question of multiplicities for future work.
Let $X:=\TT^r\setminus\cup_{i=1}^s E_i$, so that $f$ induces
a morphism $X\to Y\subset\TT^s$. We can thus apply
Theorem~\ref{tropforncc} to compute $\cT(Y)$.
We introduce the following notation and terminology.
For any compactification $X\subset\oX$ and any irreducible boundary divisor $D$,
let $[D]$ be the vector $[\val_D(f_1),\ldots,\val_D(f_s)]\in\bQ^s$.
By Theorem~\ref{CorOfMain}, $[D]\in\cT(Y)$ for any $D$.
We define the simplicial complex $\Delta(\oX)$ that describes the combinatorics of
the boundary of $\oX$ as in Theorem~\ref{tropfornccmain}.
As in Remark~\ref{importantremark}, 
the tropical variety $\cT(Y)$ is contained in the union of cones
with rays given by $[D]$ and cones described by $\Delta(\oX)$.
We say that $\oX$ {\em computes} $\cT(X)$ if $\cT(Y)$ is equal to this union of cones.
For example, if $\oX$ is smooth with a normally crossing boundary then $\oX$ 
computes $\cT(X)$ by Corollary~\ref{tropforncc}.
But Proposition~\ref{skeleton} suggests that it may suffice to assume
that $\oX$ has ``combinatorial normal crossings'', i.e.~to assume that $k$ 
boundary divisors intersect in codimension~$k$.
Below we describe how this works
 for surfaces, the case $r=2$. In what follows we assume $r=2$.

We construct a compactification $\oX$ of $X$ in two steps.
First we compactify $X$ by a toric surface $\bP(\tilde\cN)$, where $\tilde\cN$ is some
strictly simplicial refinement of the normal fan $\cN$ of the Minkowski sum
of Newton polygons of $f_1,\ldots,f_r$. By Theorem~\ref{CorOfMain}, 
the rays of $\cT(Y)$ that we see on this stage are $[E_i]=e_i$ for $i=1,\ldots,s$
and $[D_j]=\bQ_{\ge0}\Psi(\rho_j)$, where $\rho_i\subset\tilde\cN$ for $j=1,\ldots,p$ 
are the rays that correspond to the toric divisors $D_j$.

Consider a morphism $\oX\to\bP(\tilde\cN)$ which is
a composition of blow-ups with smooth centers.
Let $\tilde E_1,\ldots,\tilde E_s$,  $\tilde D_1,\ldots,\tilde D_p$ be 
the proper transforms
and $F_1,\ldots,F_q$ the exceptional divisors.
We determine the simplicial complex $\Delta(\oX;E,D,F)$ as in Theorem~\ref{tropfornccmain}.
Let $u_{ij}$ (resp.~$v_{ij}$)
be the coefficient of $F_i$ in the pullback of $E_j$ (resp.~$D_j$).
Note that
on $\bP(\tilde\cN)$ the divisor of the rational function $f_i$ is given by
$(f_i)=E_i+\sum_j\Psi_i(\rho_j)D_j$. It follows that on $\oX$ we have
$$(f_i)=\tilde E_i+\sum_j\Psi_i(\rho_j)\tilde D_j+\sum_k
\bigl(u_{ki}+\sum_j\Psi_i(\rho_j)v_{kj}\bigr)F_k.$$
It follows that the vectors $[F_k]$ in $\QQ^s$ are given by the formula
\begin{equation}\label{pullback}
[F_k]= (u_{k1},\ldots,u_{ks}) +\Psi\left(v_{k1}\rho_1 + \cdots + v_{kp} \rho_p \right).
\end{equation}

\begin{proposition}
Let $r=2$ and let $\oX\to \bP(\cN')$ be a morphism, constructed 
as a composition of blow-ups of points,
such that no three divisors from the collection $\tilde E$, $\tilde D$, and $F$
have a common point. Then $\oX$ computes the tropical surface $\cT(Y)$.
\end{proposition}

\begin{proof}
It suffices to observe, using \eqref{pullback}, that new exceptional divisors added by a resolution
of singularities $\oX'\to\oX$ do not  change the tropicalization:
each time we blow-up a point $p$ that belongs to a unique boundary divisor $\Delta$,
we obviously have $[F_p]=[\Delta]$ and the only new simplex created is a pair
$\{F_p,\Delta\}$, which gives rise to the same ray $[\Delta]$ on the tropicalization.
Similarly, the blow-up of a point $p$ that belongs to two boundary divisors $\Delta_1$ and 
$\Delta_2$ creates a new ray $[F_p]=m_1[\Delta_1]+m_2[\Delta_2]$,
where $m_i$ is the multiplicity of $\Delta_i$ at $p$, and the tropicalization
stays the same. 
\end{proof}

\bigskip

\noindent {\bf Authors' addresses:}

\medskip

\noindent  Bernd Sturmfels,
  Department of Mathematics, University of California,
   Berkeley, CA 94720-3840, USA,
{\tt bernd@math.berkeley.edu}

\medskip

\noindent  Jenia Tevelev,
  Department of Mathematics, University of Massachusetts,
Amherst, MA 01003-9305, USA,
{\tt tevelev@math.umass.edu}

\end{document}